\documentclass[12pt,a4paper]{article}
\usepackage{amsmath,amsthm,amsfonts,amssymb}
\usepackage{lscape}
\usepackage{float}
\usepackage[caption = false]{subfig}
\usepackage[demo]{graphicx}
\usepackage{epstopdf}
\usepackage[mathscr]{euscript}
\usepackage{enumerate}

\usepackage{float}
\usepackage{hyperref}

\newtheorem{thm}{Theorem}[section]

\newtheorem{lem}{Lemma}[section]

\newtheorem{defn}[thm]{Definition}

\newtheorem{rem}[thm]{Remark}

\newtheorem{conje}[thm]{Conjecture}
\begin{document}
\title{On periods of Herman rings and relevant poles}
\author{Subhasis Ghora\footnote{School of Basic Sciences, IIT Bhubaneswar, Bhubaneswar, India (sg36@iitbbs.ac.in).}\footnote{Corresponding author.} 
\and Tarakanta Nayak\footnote{School of Basic Sciences, IIT Bhubaneswar, Bhubaneswar, India (tnayak@iitbbs.ac.in).} 
 }

\date{}
\maketitle

		\begin{abstract}
	
Possible periods of Herman rings  are studied for general meromorphic functions with at least one omitted value. A pole is called $H$-relevant for a Herman ring $H$ of such a function $f$ if it is surrounded by some Herman ring of the cycle containing $H$. In this article, a lower bound on the period $p$ of a Herman ring $H$ is found in terms of the number of  $H$-relevant poles, say $h$. More precisely, it is shown that $p\geq \frac{h(h+1)}{2}$ whenever $f^j(H)$, for some $j$, surrounds a pole as well as the set of all omitted values of $f$.  It is proved that $p \geq \frac{h(h+3)}{2}$ in the other situation. Sufficient conditions are found under which equalities hold.
			It is also proved that if an omitted value is contained in the closure of an invariant or a two periodic Fatou component then the function does not have any Herman ring. 
		 
		\end{abstract}
	
\textit{Keyword:}
 Omitted values, Herman rings and Transcendental meromorphic functions.\\
 Mathematics Subject Classification(2010) 37F10,  37F45

	\section{Introduction}

Let $f:\mathbb{C}\rightarrow \widehat{\mathbb{C}}$ be a transcendental meromorphic function such that it has either at least two poles or exactly one pole which is not an omitted value. For such functions, there are infinitely many points whose iterated forward image is infinity, the only essential singularity of the function. This is the reason why these are called general meromorphic functions. The class of all such functions is denoted by $M$ in the literature \cite{Ber93a}. A family of meromorphic functions defined on a domain is called normal if each sequence taken from the family has a subsequence that converges uniformly  on every compact subset of the domain. The limit is allowed to be infinity. 
The Fatou set of $f$ is the set of all points in a neighborhood of which the family of functions  $\{f^{n}\}_{n>0}$ is well defined and normal. For general meromorphic functions, normality is in fact redundant. More precisely, the Fatou set of a general meromorphic function is the set of all points where $f^n$ is defined for all $n$ \cite{Ber93a}. Its complement is the Julia set and is denoted by $\mathcal{J}(f)$. 

  The Fatou set is open by definition. A maximal connected subset of the Fatou set is called a {Fatou component}. For a {Fatou component} $U$ and a natural number $k$, let $U_{k}$ denote the {Fatou component} containing $f^{k}(U)$. A {Fatou component} $U$ is called $p$-periodic if $p$ is the smallest natural number satisfying $U_p = U $. We say $U$ is invariant if $p=1$. The connectivity of a periodic Fatou component is known to be $1,~2$ or $\infty$ \cite{Ber93a}. The sequence of iterates $f^n$ has finitely many limit functions on a periodic Fatou component. Depending on whether such limit functions are constants or not, a periodic Fatou component can be an attracting domain, a parabolic domain, a Baker domain, a Siegel disk, or a Herman ring. The last two possibilities arise precisely when the limit functions are non-constant. This article is mainly concerned with Herman rings. 
 
   \par 
    A $p$-periodic Fatou component $H$ is called a Herman ring if there exists an analytic homeomorphism $\phi : H \rightarrow \{z:1<|z|<r\}$ such that $f^p$ is conformally conjugate to an irrational rotation. In other words, $\phi(f^p(\phi^{-1}(z)))=e^{i 2 \pi \alpha}z$ for some irrational number $\alpha$ and for all $z, 1<|z|<r$. Clearly, a Herman ring is doubly connected.
    
    \par 
   Using the Maximum Modulus Principle, it can be shown that  transcendental entire functions cannot have any Herman ring. However, general meromorphic functions are known to have Herman rings. A meromorphic function of finite order can have at most finitely many Herman rings whereas, there are transcendental meromorphic functions having infinitely many Herman rings. This is proved by Zheng in~ \cite{zheng}. It is known that if a transcendental meromorphic function has $N$ poles then it has at most $N$ invariant Herman rings \cite{peter}. Dominguez and Fagella showed that, for a given $N>0$, there exists an $f \in M$ with exactly $N$ poles and $N$ invariant Herman rings \cite{fagellaD}. Herman rings for general meromorphic function satisfying other conditions are also constructed by the authors. 
    
\par
 A Herman ring is always doubly connected giving rise to a disconnected Julia set. In other words, a connected Julia set ensures the non-existence of Herman rings. Baranski and co-authors proved that transcendental meromorphic functions arising as Newton maps of  entire functions have connected Julia sets, and hence have no Herman ring \cite{Bar2018}. Another class of functions, namely those general meromorphic functions omitting at least one value is studied by Nayak \cite{Nayak2016} and, Nayak and Zheng \cite{tk}. A value $z_{0}\in \widehat{\mathbb{C}}$ is said to be an \textit{omitted value} of a function $f$ if $f(z)\neq z_{0}$ for any $z\in\mathbb{C}$. A number of sufficient conditions guaranteeing the non-existence of Herman ring are provided by the authors. They proved the following results. If all the poles of such a function are multiple, then it has no Herman ring. Functions with a single pole or with at least two poles, one of which is an omitted value, have no Herman
ring. Examples of functions which has no Herman ring are also provided in ~\cite{small}. In view of all these, following conjecture can be made.
      
      \begin{conje}
     If a general meromorphic function omits at least one point in the plane then it does not have any Herman ring.
      \label{nohermanring}
      \end{conje}
      This is the motivation for the current work. 
       
 \par   

 Let $O_f$ denote the set of all omitted values of $f$. Note that $O_f$ consists of at most two points, and is a subset of the plane whenever $f \in M$. Let $M_o$ be the set of all functions in $M$ having at least one omitted value. All functions considered in this article belong to $M_o$.

Nayak has proved that, if $f \in M_o$ then $f$ has no Herman ring of period $1$ or $2$~\cite{Nayak2016}.  The proof contains a detailed analysis of the possible arrangements of Herman rings in the plane relative to each other. We say a set is surrounded by a Herman ring $H$ if the set is contained in the bounded component of the complement of $H$. The locations of the omitted value(s) and poles surrounded by Herman rings have also been key to a number of useful observations. Later, these observations are used to show that there cannot be more than one $p$-cycles of Herman rings for $p=3,4$ \cite{small}. These ideas are developed and used in this article to prove a lower bound for periods of Herman rings and non-existence of the same under certain situation.

Given a Herman ring $H$ of $f$, a pole $w$ is said to be $H$-relevant if some Herman ring, in short ring $H_i$ of the cycle containing $H$ surrounds $w$. It is important to note that every cycle of Herman rings contains at least one ring surrounding some pole Lemma $2.1$, \cite{tnayak-2015}.
A lower bound on the period $p$ of a Herman ring $H$ is found in terms of $h$, the number of $H$-relevant poles. More precisely, it is shown that $p\geq \frac{h(h+1)}{2}$ whenever the basic nest (See Section 2 for definition) surrounds a pole. Less technically, this condition is equivalent to the statement that $H_j$, for some $j$, surrounds a pole as well as $O_f$.  It is proved that $p \geq \frac{h(h+3)}{2}$ in the other situation, when the basic nest does not surround any pole. This is the statement of Theorem~\ref{lowerbound}. The innermost ring $H_1$ with respect to the set $O_f$ is the ring which surrounds $O_f$ but does not surround any other ring of the same cycle.
  It follows from Lemma~\ref{basiclemma} (which originally appeared in ~\cite{tk}) that $H_{1+k}$, the ring containing $f^k (H_1)$ surrounds a pole of $f$ for some $k$.
  The smallest such natural number is what we refer as the length of the basic chain. It is seen that $h$, the number of $H$-relevant poles is at most the length of the basic chain (Lemma~\ref{poles-lengthofbc}). When these two numbers are same or differ by exactly $1$, we are able to prove equality in Theorem~\ref{lowerbound} under some additional condition. This is given in Theorem \ref{lowerbound-equality}. A ring is called outermost if it is not surrounded by any other ring (Section $2$ can be seen for definition) of the cycle. Note that for a cycle of Herman rings, there can be more than one outermost ring. Theorem~\ref{lowerbound-equality} proves the following.
   If each ring surrounding a pole is outermost then $(i)$ $p=\frac{h(h+1)}{2}$ when the length of the basic chain is $h$, and $(ii)$ $p= \frac{h(h+3)}{2}$ when the length of the basic chain is $h+1$ and the basic nest does not surround any pole (equivalently, there is no ring surrounding $O_f$ as well as a pole of $f$). It is worth noting that the assumption of $(i)$ ensures that the basic nest surrounds a pole. The condition that each ring surrounding a pole is the outermost ring  of the nest  is satisfied whenever the period of a Herman ring is $3$ (\cite{small}). Theorem~\ref{periodicfatoucomponent} proves that if an omitted value is contained in the closure of a periodic Fatou component $U$ of $f$ and $f$ has a Herman ring $H$ then the number of $H$-relevant poles is strictly less than the period of $U$.
  This leads to non-existence of Herman rings whenever $U$ is invariant or $2$-periodic.  This is a new condition under which Conjecture~\ref{nohermanring} is true.

  \par  
  Section $2$ discusses all the preliminary ideas and known results required for the proofs in the next section. All the new results are stated and proved in Section $3$.
\par    
We reserve the notation $f$ for functions in $M_o$ throughout this article. By a ring, we mean a Herman ring in this article. For a ring $H$, let $B(H)$ denote the bounded component of the complement of $H$. We say $H$ surrounds a set $A$ (or a point $w$) if $A \subset B(H)$ (or if $w \in B(H)$ respectively). For a $p$-periodic Herman ring $H$, denote the cycle of $H$ by $\{H_0,~H_1,\dots ,H_{p-1}\}$, where $H=H_0=H_p$. In view of Theorem $1.3$, \cite{Nayak2016} which states that $f \in M_o$ has no Herman ring of period $1$ or $2$, we assume that $p >2$ throughout.

\section{Preliminaries}
 A Jordan curve in a multiply connected Fatou component of a  meromorphic function can be considered such that it is not contractible in the Fatou component.
  Since the backward orbit of $\infty$ does not intersect the Fatou set, $f^n$ is well defined for all $n$ on such a Jordan curve. The following lemma, proved in~\cite{tk} analyzes the iterated forward images of such a Jordan curve leading to useful conclusions.  
\begin{lem}
Let $f \in M$ and $V$ be a multiply connected Fatou component of $f$. Also let $\gamma$ be a non-contractible closed curve in $V$ (that means $B(\gamma)\cap \mathcal{J}(f) \neq \phi$). Then there exists an $n \in \mathbb{N} \cup \{0\}$ and a closed curve $\gamma_n \subset f^n(\gamma)$ in $V_n$ such that $B(\gamma_n)$ contains a pole of $f$. Further, if $O_f \neq \emptyset$, then $O_f \subset B(\gamma_{n+1})$ for some closed curve $\gamma_{n+1}$ contained in $f(\gamma_n)$.  
\label{basiclemma}
\end{lem}
That a multiply connected Fatou component corresponds to a pole follows from the above lemma. A Herman ring is doubly connected. Above lemma applied to a Herman ring gives rise to the following.
\begin{rem}
Let $H$ be a $p$-periodic Herman ring of $f$ and\\ $\phi : H \to A=\{z :1 < |z|<r\}$ be the analytic homeomorphism such that $\phi(f^p(\phi^{-1}(z)))=e^{i 2 \pi \alpha}z$ for some irrational number $\alpha$ and for all $z \in A$. If $\gamma$ is the pre-image of a circle $\{z: |z|=r' \}$ for $1 < r'< r$ under $\phi$ then $\gamma$ is an $f^p$-invariant and non-contractible Jordan curve in $V$ and the set $\{\gamma_n := f^{n} (\gamma): n>0\}$ is a finite set of $f^p$-invariant Jordan curves. Further, there is a $j$ such that $\gamma_j$ surrounds a pole of $f$.
\label{rem-basiclemma}\end{rem} 
Recall that for a $p$-periodic Herman ring $H$, the cycle of $H$ is denoted by $\{H_0,~H_1,\dots ,H_{p-1}\}$, where $H=H_0=H_p$.  All the definitions given and used in this article are with respect to $H$.  The next two definitions were introduced in \cite{Nayak2016}.

\begin{defn}{\bf{(H-relevant pole)}}\\
Given a Herman ring $H$, a pole $w$ is said to be $H$-relevant if some ring $H_i$ of the cycle containing $H$ surrounds $w$.
\end{defn}

It is clear from Remark~\ref{rem-basiclemma} that an $f^p$-invariant Jordan curve surrounds a pole whenever $f$ has a $p$-periodic Herman ring. Since the curve is in a ring, the existence of at least one $H$-relevant pole is evident.  A refinement of this statement is implicit in the following theorem, proved in ~\cite{Nayak2016}.
\begin{thm}
If $f \in M_o$ has only one pole, then $f$ has no Herman ring.
\end{thm}

 It follows from the above theorem that the number of $H$-relevant poles is at least $2$ for every function $f \in M_o$. 
 
The position of rings relative to each other is going to play an important role in our investigation.
\begin{defn}{\bf{(H-maximal nest)}}\\
Given a Herman ring $H$, a ring $H_j$ is called an $H$-outermost ring if $H_i$ does not surround $H_j$ for any $i$, $i \neq j$. Given an  outermost ring $H_j$, the collection of all rings consisting of $H_j$ and all those in the same cycle surrounded by $H_j$ is called an $H$-maximal nest. We call it simply a nest whenever $H$ is understood from the context.
\end{defn}
 A nest is a sub-collection of Herman rings from the periodic cycle containing $H$. Each $H_i$ belongs to exactly one nest.
By saying a nest surrounds a point (or a set), we mean the outermost ring of the nest surrounds the point (or the set respectively). This is also true whenever any other ring  belonging to the nest surrounds the point or the set. It is important to note that a nest surrounds at most one pole. This follows from Lemma 2.4, ~\cite{Nayak2016}, which is stated below.
 \begin{lem}
 If $H$ is a Herman ring of $f \in M_o$, then $f:B(H)\rightarrow  \widehat{\mathbb{C}}$ is one-one.
 \label{one-one}
 \end{lem}
 It follows from the above lemma that, if all the poles of a function belonging to $M_o$ are multiple, then it has no Herman ring (Corollary $2.6$, \cite{Nayak2016}). Note that $f$ is not one-one in the plane even though it is so in $B(H)$ for every Herman ring $H$.

If a ring $H_i$ surrounds a pole then consider a non-contractible $f^p$-invariant Jordan curve $\gamma_i$ in $H_i$. Now, $\gamma_i$ surrounds the pole and it follows from the last part of Lemma~\ref{basiclemma} that $f(\gamma_i)$ surround $O_f$. In other words, $H_{i+1}$ surrounds $O_f$. The nest containing $H_{i+1}$ is too important to have a name.
 \begin{defn}{\bf (Basic nest)}
 Given a Herman ring $H$, the $H$-maximal nest surrounding the set of all omitted values of $f$ is called the basic nest of $H$. A nest different from the basic nest is called non-basic.
 \end{defn}
Here is a useful remark.
\begin{rem}
If a ring of a cycle of Herman rings surrounds a pole then its image surrounds $O_f$ and therefore the image is in the basic nest. In other words, the periodic pre-image of each ring of a non-basic nest does not surround any pole. 
\label{rem-basicnest}
\end{rem}

We need the idea of innermost rings for making some new definitions.

\begin{defn}{\bf{(Innermost ring with respect to a set)}}
Given a Herman ring $H$, we say a ring $H_j$ is innermost with respect to a set $S$ if $H_j$ surrounds $S$ but not $H_i$  for any $i, i \neq j$.
\end{defn}
A ring is said to be innermost in a nest if it does not surround any other ring of the nest.
Existence of more than one innermost ring in a nest cannot be ruled out. 
\begin{defn}
 {\bf{   (Basic chain and Basic rings})}
Given a Herman ring $H$, the ordered set of rings $\{H_1, H_2, H_3,\dots, H_n\}$ is called the basic chain, where $H_1$ is the innermost ring with respect to $O_f$ and $n$ is the smallest natural number such that $H_{n}$ surrounds a pole $w$. Each ring $H_i,~ 1 \leq i \leq n$ is said to be a basic ring of $H$ and $w$ is said to be the pole corresponding to the basic chain.
\end{defn}
Now onwards, we reserve $H_1$ to denote the innermost ring with respect to $O_f$. 
 The ring $H_1$ does not surround any pole by Remark $2.10$ of ~\cite{small}. Thus, in view of Remark \ref{rem-basiclemma}, the number of basic rings is at least $2$. Here is a useful remark on $H_1$ following from the periodicity of Herman rings.
 \begin{rem}
 For each ring $H'$ in a non-basic nest, there is a $j$ such that $H_{1+j} =H'$. 
 \label{basic-nonbasicnest}
 \end{rem}
Instead of starting from the innermost ring with respect to $O_f$ as is done in the definition of the basic chain, one can start from any ring $H_r$  surrounding $O_f$ ( but not any pole ) and look at the smallest $m$ for which $H_{r+m}$ surrounds a pole. This gives rise to the following definition.
\begin{defn}
 {\bf{   (Chain})}
 The ordered set of rings $C= \{H_r, H_{r+1},\dots, H_{r+m}\}$ is called a chain if $H_r$ is a ring surrounding $O_f$ but not any pole, and $m$ is the smallest natural number such that $H_{r+m}$ surrounds a pole.  The number  of rings in a chain $C$ is called its length, and we denote it by $|C|$.
\end{defn}

It is clear that two chains are same or disjoint. Note that the basic chain is the unique chain whose first ring is $H_1$. It is of course a basic ring. 
Further, the first ring of every chain belongs to the basic nest and the length of every chain is at least two. It is important to note that the last ring of a chain $C$ surrounds a pole, say $w$. We say $C$ corresponds to $w$.  Though the possibility of two chains corresponding to the same pole cannot be ruled out, chains corresponding to different poles are important for our purpose.
\begin{defn}{\bf (Independent chains)}
Two chains are called independent if they correspond to two different poles.
\end{defn}
 Here are two basic observations on the length of chains.
  
  \begin{lem}\
  \begin{enumerate}
  \item The length of every chain is less than or equal to that of the basic chain.
  \item If $C_i$ and $C_j$ are two independent chains then their lengths are different.
  \end{enumerate}
  \label{basicchain-properties}
  \end{lem}
  \begin{proof}
   \begin{enumerate}
   \item  Let $C=\{ H_{r+1}, H_{r+2}, \cdots, H_{r+n} \}$ be a chain different from the basic chain. Then  $H_{r+i}$ does not surround any pole for $i=1, 2,\dots, {n-1}$. Since $C$ is different from the basic chain, $H_{r+1}$ surrounds $H_1$, the innermost ring with respect to $O_f$. Since $H_1$ is the first ring of the basic chain, it follows from the Maximum Modulus Principle that $H_{r+i}$ surrounds $H_{i}$ for each $i = 1, 2, \cdots, n $.
  The pole corresponding to the basic chain is surrounded by either $H_{n}$ or $H_{k}$ for some $k >n$. This gives that the length of every chain is less than or equal to that of the basic chain.
   
   \item  Let $C_{i}$ and $ C_{j}$ be two independent chains. By definition of independent chains, the poles $w_i$ and $w_j$ corresponding to $C_{i}$ and $C_{j}$ respectively are different. Let $H_{i+1}$ and $H_{j+1}$ be the initial rings of $C_{i}$ and $C_{j}$ respectively. Then both of these surround $O_f$ and hence, either $H_{i+1}\subseteq B(H_{j+1})$ or  $H_{j+1}\subseteq B(H_{i+1})$. Without loss of generality, let $ H_{i+1}\subseteq B(H_{j+1})$. If the length of $C_i$ and $C_j$ is the same, say $l $, then $H_{j+l}$ surrounds $H_{i+l}$ which gives that both are contained in the same nest. Also $H_{i+l}$ and $H_{j+l }$ surround the poles $w_i$ and $w_j$ respectively. But each nest surrounds at most one pole (by Lemma~\ref{one-one}) giving that $w_i=w_j$ which is a contradiction. This proves that the length of $C_i$ is different from that of $C_j$.
   \end{enumerate}
  \end{proof}
  
\section{Results and their proofs}

The basic chain, introduced in the previous section is going to play a key role in the proofs. To start with, we make an observation on how it restricts  the number of nests in a cycle of Herman rings.

\begin{lem}
Let $H$ be a $p$-periodic Herman ring of $f$. Then the number of nests in the cycle of $H$ is at most the length of the basic chain.
\label{nests-lengthofbs}
\end{lem}
\begin{proof}
We first show that each nest contains at least one basic ring. This is clearly true for the basic nest. Now let $N$ be a non-basic nest. Recall that $H_1$ is the innermost ring with respect to $O_f$ and it does not surround any pole. The ring $H_1$ is not in $N$ and one of its iterated forward image is in $N$ by Remark~\ref{basic-nonbasicnest}.
 Let $n$ be the smallest natural number such that $H_{n}$ is in $N$. We assert that $H_n$ is a basic ring.
 
 The ring $H_{n-1}$, the periodic pre-image of $H_n$, does not surround any pole by Remark~\ref{rem-basicnest}. If $H_{k}$ surrounds a pole for some $k, ~1 < k < n$ then consider the largest such $k$ and denote it by $k^*$. As observed in the previous paragraph, $k^* \neq n-1$.  Therefore $2 \leq k^* \leq n-2$. It follows from  Remark \ref{rem-basicnest} that $H_{k^*+1}$ surrounds $O_f$ and hence, is in the basic nest. Further, none of $H_{k^*+1}, H_{k^*+2},\dots, H_{k^*+n-k^*-1}=H_{n-1} $ surrounds any pole by the choice of $k^*$. Since $H_{1}$ is the innermost ring with respect to $O_f$, either $H_{k^*+1}$ surrounds $H_1$ or is equal to $H_1$. It follows from Lemma~\ref{one-one} and the Maximum Modulus Principle that $ H_{k^*+j}$ surrounds or is equal to $H_{j}$ for each $j,~ 1 \leq j \leq n-k^*$. Further, the map 
    $f^{n-k^*-1}: B(H_{k^*+1}) \to B(H_n)$ is conformal. This gives that $H_{n-k^*}$ is a  ring in the nest $N$. However, this contradicts our earlier assumption that $n $ is the smallest natural number such that $H_n$ is in $N$.  This proves that $H_{k}$ does not surround any pole for any  $k,~ 1 < k < n $ and hence $ H_{n}$ is a basic ring.  
     \par   
    
     Since two different nests cannot contain the same basic ring, the number of nests is at most the number of basic rings. The proof is completed by noting that the number of basic rings is nothing but the length of the basic chain.

\end{proof}
  The number of $H$-relevant poles is at least two by Lemma 2.11 of \cite{small}. An upper bound for this number can be obtained using the previous lemma.

\begin{lem}
    Let $H$ be a Herman ring of $f$. Then the number of $H$- relevant poles is at most the length of the basic chain.
\label{poles-lengthofbc}
\end{lem}

\begin{proof}

 It follows from Lemma \ref{nests-lengthofbs} that the total number of nests in the cycle of Herman rings is less than or equal to the length of the basic chain. Further, each nest surrounds at most one $H$-relevant pole by Lemma~\ref{one-one}. This gives that the number of $H$-relevant poles is at most the number of nests. Hence the number of $H$-relevant poles is at most the length of the basic chain.
\end{proof}
\begin{rem} 
Let $h, n$ and $l$ denote the number of $H$-relevant poles, the number of nests and the length of the basic chain corresponding to a cycle of Herman rings respectively. Then it follows from the proof of the above lemma that $h \leq n \leq l$.  If $h = l$ then $h = n = l$. It is evident from the proof of Lemma~\ref{nests-lengthofbs} that each nest contains at least one basic ring. Therefore, each nest contains exactly one basic ring in this case. Also each nest contains an $H$-relevant pole. As evident from  Lemma $3.1$ of \cite{small}, this is the case if the period of the Herman ring is three.\label{hrp-n-lbc}
\end{rem}
\begin{rem}
If there is a $4$-periodic Herman ring, it is seen (Lemma~3.2,  ~\cite{small}) that the length of the basic chain is always three whereas the number of $H$-relevant poles is always two.
Further, the number of nests can be two or three.
\end{rem} 

Note that two different chains do not contain a common ring.  Since the number of all rings in a cycle of Herman rings is the period of the cycle, the lengths of chains are crucial. Next result determines the number of independent chains and their lengths in terms of the number of $H$-relevant poles. Recall that any two independent chains have different lengths by Lemma \ref{basicchain-properties}$(2)$. 

\begin{thm}
    Let $H$ be a $p$-periodic Herman ring of $f$ and $h$ be the number of $H$-relevant poles. Then the number of independent chains is  $h-1$ or $h$. If the basic nest does not surround any pole then the number of independent chains is $h$.
    If $  c \in \{h-1, h \}$, and $C_2, C_3, \dots, C_{c+1}$  are the independent chains such that $|C_2| < |C_3|<  \dots <|C_{c+1}| $  then $|C_j|\geq j$ for all $j, ~2 \leq j \leq c+1$.
\label{independentchain}
\end{thm}
\begin{proof}
We first show that for each non-basic nest surrounding a pole, there is a chain corresponding to that pole.
Let $N$ be a non-basic nest surrounding a pole. 
This means that a ring, say $H_r$ belonging to $N$ surrounds a pole $w$ of $f$.
 Such a nest exists as there are at least two $H$-relevant poles. 
The ring $H_{r-1}$, the periodic pre-image of $H_r$ does not  surround any pole by Remark~\ref{rem-basicnest}.
  Since $H$ is periodic, there is an $m$ such that $H_{r-m}$, the periodic pre-image of $H_r$ under $f^{m}$, surrounds a pole. 
  Choose the smallest such $m$ and observe that $m  \geq 2$. 
  Then $H_{r-m+1}$ does not surround any pole. Further, the ring $H_{r-m+1}$ is in the basic nest by Remark~\ref{rem-basicnest}. 
  Thus $\{H_{r-m+1}, H_{r-m+2}, \dots , H_{r} \}$ is a chain.
 Therefore, 
 
 for each non-basic nest surrounding a pole $w$, there is a chain corresponding to $w$.  \hspace{11.5cm}                     $(*)$ 
 
 Each $H$-relevant pole, except possibly one is surrounded by a non-basic nest. In other words, the number of non-basic nests surrounding some pole is either $h-1$ or $h$. It follows from $(*)$ that the number of independent chains is  either $h-1$ or $h$. If the basic nest does not surround any pole then the number of different non-basic nests surrounding some pole is $h$. This is nothing but the number of independent chains. 
  \par   
Since lengths of two different independent chains are different (Lemma~\ref{basicchain-properties}), the chains can be ordered according to their lengths. Let $|C_2| < |C_3|< \dots <|C_{c+1}|$. Suppose that $|C_j|<j$ for some $j,~ 2 \leq j \leq c+1$.  Note that the length of each chain is at least two. In particular $|C_2| \geq 2$, which gives that $2 \leq |C_2| < |C_3|< \dots <|C_{j}| \leq j-1$.  However, this is not possible by the Pegionhole Principle. Therefore, $|C_j|\geq j$, for all $j=2~, 3, \dots, ~c+1$.
     
\end{proof}

The situation when the basic nest does not surround any pole is dealt with in Theorem \ref{independentchain}. Following is a remark on the other case.
\begin{rem}
\begin{enumerate}
\item If the basic nest surrounds a pole $w$ then the number of independent chains is $h$ or $h-1$.

\item
Let $w$ be a pole surrounded by the basic nest. If each ring of the basic nest surrounding $w$ also surrounds $O_f$, then the number of independent chains is $h-1$. To prove it, note that there are $h-1$ independent chains corresponding to each $H$-relevant pole surrounded by a non-basic nest (it follows from $(*)$). In order to show that those are the only independent chains, suppose on the contrary that there is a chain $\{ H_r,  H_{r+1}, \dots, H_{k-1}, H_k \}$ corresponding to $w$. Then $H_k$ surrounds $w$ as well as $O_f$. Since $H_{k-1}$ does not surround any pole (by definition of chain), $f:B(H_{k-1}) \to B(H_k)$ is conformal. This is a contradiction as $B(H_k)$ contains $O_f$ and $B(H_{k-1})$ is bounded (See Lemma 2.1, \cite{tnayak-2015}). Hence there is no chain corresponding to $w$. Therefore, the number of independent chains is exactly $h-1$  whenever each ring surrounding $w$ also surrounds $O_f$. 
\end{enumerate}
\label{basicnest-pole}
\end{rem}
We now prove the first main result of this article.

\begin{thm}{\bf{(Lower bound for the period of Herman ring)}}
    Let $H$ be a $p$-periodic Herman ring  of a function $f$ and $h$ be the number of $H$-relevant poles. Then $p \geq \frac{h(h+1)}{2}$. In particular, the following are true.
\begin{enumerate}
\item If the basic nest surrounds a pole then $p\geq \frac{h(h+1)}{2}$.
\item If the basic nest does not surround any pole then $p\geq \frac{h (h+3)}{2}$.
\end{enumerate}
\label{lowerbound}
\end{thm}
\begin{proof}
Note that two rings belonging to two different independent chains are different. 
\begin{enumerate}
\item 
It follows from Remark \ref{basicnest-pole}(1) that the number of independent chains is either $h$ or $h-1$. We deal with these cases seperately.
\par Let the number of independent chains be $h-1$. Note that the basic nest surrounds a pole and therefore, there are $h-1$ many $H$-relevant poles surrounded by non-basic nests. It follows from $(*)$ that there are $h-1$ chains, each corresponding to one such pole. Clearly, these chains are independent and are the only independent chains. Let the $h-1$ chains be denoted by $C_2,~ C_3,~ \dots, ~C_{h}$. The total number of rings contained in all these independent chains is $|C_2|+|C_3|+\cdots+|C_{h}|$, which is at least $  2+3+\cdots+h = \frac{h(h+1)}{2}-1$ by Theorem \ref{independentchain}. Now, the ring in the basic nest surrounding a pole can not be a part of any chain, because this chain along with the previously found $h-1$ chains would be independent contradicting our assumption. Therefore, the total number of rings in the cycle containing $H$ is at least $1$ more than $\frac{h(h+1)}{2}-1$. Thus $p \geq \frac{h(h+1)}{2} $.

\par

If the number of independent chains is $h$ then the total number of rings contained in all these chains is $|C_2|+|C_3|+\cdots+|C_{h+1}|$
. This number is at least $2+3+\cdots+(h+1)=\frac{h(h+3)}{2}$ by Theorem~\ref{independentchain}. Therefore, $p \geq \frac{h(h+3)}{2}$ which is clearly bigger than  $\frac{h(h+1)}{2}$.

\item If the basic nest does not surround any pole, then there exists $h$ number of independent chains by $(*)$. Arguing as in the previous paragraph, it follows that the total number of rings in the cycle of $H$ is greater than or equal to
     $ |C_2|+|C_3|+\cdots+|C_{h+1}|\geq 2+3+\cdots+(h+1)=\frac{h(h+3)}{2} $. Thus $p\geq \frac{h(h+3)}{2}$.
\end{enumerate}
\end{proof}
An additional assumption leads to equality in Theorem~\ref{lowerbound}. The assumption is satisfied for $3$-periodic Herman rings \cite{small}. Recall that $h$ denotes the number of $H$-relevant poles.
\begin{thm}
Let $H$ be a $p$-periodic Herman ring and each ring surrounding an $H$-relevant pole be the outermost ring of the concerned nest. 
\begin{enumerate}
\item If the length of the basic chain is equal to $h$ then $p=\frac{h (h+1)}{2}$.
\item If the length of the basic chain is equal to $h+1$ and the basic nest does not surround any pole then $p=\frac{h (h+3)}{2}$.
\end{enumerate}
\label{lowerbound-equality}
\end{thm}
\begin{proof}
We assert that the set of all chains is the same as the set of all independent chains. Equivalently, two different chains are independent. To prove this by contradiction, suppose that $\{H_r,~H_{r+1},\cdots, ~H_k \}$ and $\{H_i,~ H_{i+1},\cdots, ~H_{k'} \}$ are two different chains and are not independent. Then they correspond to the same pole. In other words, $H_k$ and $H_{k'}$ are different rings surrounding the same pole. However, this negates our assumption that only the outermost ring of a nest surrounds a pole. 

The period of $H$ is to be determined by finding the lengths of all independent chains.

\begin{enumerate}
\item 
Since the number of $H$-relevant poles is $h$, the number of independent chains is $h-1$ or $h$, by Theorem \ref{independentchain}. As the length of the basic chain is equal to $h$ then there are exactly $h$ number of nests and each nest contains exactly one basic ring. Further, each nest surrounds an $H$-relevant pole. All these statements are observed in Remark~\ref{hrp-n-lbc}. In particular, the basic nest surrounds a pole say $w$. Since by the assumption, only the outermost ring of a nest surrounds a pole, the ring in the basic nest which surrounds $w$ also surrounds the $O_f$. Thus all the conditions of Remark~\ref{basicnest-pole}(2) are satisfied. Hence the number of independent chains is $h-1$ and there is no chain corresponding to $w$.  Let the $h-1$ chains be denoted by $C_2,~ C_3,~ \dots, ~C_{h}$. Then $2 \leq |C_2| < |C_3|< \cdots < |C_h|$ by Theorem~\ref{independentchain}. Note that $C_h$, being the longest chain is the basic chain. Further, $ |C_h|=h$ by assumption. This is possible only when $|C_j|= j$ for all $j=2, 3, \dots, h$ by the Pegionhole Principle. Thus the total number of rings in the cycle of $H$ is $ 1+2+ \cdots+h=\frac{h(h+1)}{2}$. Note that the first term $1$ in the sum corresponds to the outermost ring of the basic nest. Thus $p=\frac{h(h+1)}{2}$.
\item 
It follows from Theorem~\ref{independentchain} that the number of independent chains is $h$. Let the independent chains be denoted by $C_2,~ C_3,~ \dots, ~C_{h+1}$. Then $2 \leq |C_2| < |C_3|< \cdots < |C_{h+1}|$ by Theorem~\ref{independentchain}. Note that $C_{h+1}$, being the longest chain is the basic chain. Further, $ |C_{h+1}|=h+1$ by assumption. This is possible only when $|C_j|= j$ for all $j=2, 3, \dots, h+1$, again by the Pegionhole Principle. Thus the total number of rings in the cycle of $H$ is $ 2+3+ \cdots+(h+1)=\frac{h(h+3)}{2}$ proving that $p=\frac{h(h+3)}{2}$.

\end{enumerate}
\end{proof}
\begin{rem}
The assumption in Theorem~\ref{lowerbound-equality}(1) gives that the basic nest surrounds a pole. This follows from the proof.
\end{rem}
Theorem~\ref{lowerbound} gives a lower bound for the period of a Herman ring (if exists) in terms of the number of $H$-relevant poles. The following theorem provides a lower bound for the period of a periodic Fatou component containing an omitted value in terms of the number of $H$-relevant poles.

\begin{thm}
Let  $U$ be a periodic Fatou component of $f$ such that its closure contains at least one omitted value. Then for every Herman ring $H$ of $f$, the number of $H$-relevant poles is strictly less than the period of $U$. In particular, if $U$ is invariant or $2$-periodic then $f$ has no Herman ring.
\label{periodicfatoucomponent}
\end{thm}
\begin{proof}
Let $U$ be a $q$-periodic Fatou component such that its closure contains an omitted value of $f$ and $\{U=U_1,~U_2,\dots , U_q \}$ be the cycle. Let $H$ be any Herman ring of $f$ and $h$ be the number of $H$-relevant poles. Let $\{ H_1, H_2, \dots, H_l \}$ be the basic chain. If $H_j=U$ for some $j$ then $p \geq \frac{h(h+1)}{2}$ by Theorem \ref{lowerbound}. But $h < \frac{h(h+1)}{2}$ since $h \geq 2$. This gives that $h < q$ as desired.

Now suppose that $H_j \neq U$ for any $j$. Let $\{ H_1, H_2, \dots, H_l \}$ be the basic chain. Since $H_1$ surrounds $O_f$, it surrounds $U_1$. This implies that $H_i$ surrounds $U_i$ for $i=1,~2, \dots,~l$ by the Maximum Modulus Principle. Further, all such $U_is$ are bounded.  But Lemma 2.1, \cite{tnayak-2015} states that for a given region $G$, if the closure of $f(G)$ contains an omitted value then $G$ is unbounded. Hence $U_q$ is unbounded. Therefore $l<q$. Note that $h \leq l$ by Lemma ~\ref{poles-lengthofbc}. Thus $h < q $. In other words, the number of $H$-relevant poles is strictly less than the period of $U$.
\par
If $U$ is invariant or $2$-periodic then $q=1$ or $2$ and consequently $h<2$ which is not possible as number of $H$-relevant poles is at least $2$ ( Lemma 2. 11,  \cite{small}). Thus $f$ has no Herman ring if $U$ is invariant or $2$-periodic.
\end{proof}

\begin{rem}
Examples of functions in $M_o$ with periodic Fatou components containing an omitted value can be found in ~\cite{tkmgpp-unboundedsingular, tanhexp}.
\end{rem}
\begin{rem}
If $U$ is a $3$-periodic Fatou component of $f$ containing an omitted value and $H$ is a Herman ring of $f$,  then the number of $H$-relevant poles is two by Theorem~\ref{periodicfatoucomponent}. 
It follows from the proof that $H_1$ surrounds $U$. Since one of $U, U_1, U_2$ is unbounded, the length of the basic chain is at most two. In fact, it is exactly two. Now it follows from Lemma~\ref{hrp-n-lbc} that there are only two nests. Further, each nest surrounds exactly one $H$-relevant pole. For two Herman rings $H, H'$ belonging to two different cycles, the set of $H$-relevant poles coincides with the set of $H'$-relevant poles. 
\end{rem}
\begin{rem}
It is shown in ~\cite{small} that the length of the basic chain is three for every $4$-periodic Herman ring. It follows from the previous remark that, if $f$ has a $3$-periodic Fatou component containing an omitted value then it has no $4$-periodic Herman ring.
\end{rem}

\end{document}